\documentclass[12pt]{amsart}
\usepackage{amsmath,amsthm,amsfonts,amssymb,mathrsfs}
\date{\today}

\usepackage{amsmath,amsthm,amsfonts,amssymb,mathrsfs}
\usepackage{eucal}

\input xymatrix
\xyoption{all}

\usepackage{color}

\usepackage{amssymb,cite}
\usepackage[colorlinks,plainpages,citecolor=magenta, linkcolor=blue, backref]{hyperref}

\usepackage{hyperref}

\newtheorem{theorem}{Theorem}[section]

\newtheorem{proposition}[theorem]{Proposition}

\newtheorem{lemma}[theorem]{Lemma}
\theoremstyle{definition}
\newtheorem{definition}[theorem]{Definition}
\newtheorem{remark}[theorem]{Remark}

\begin{document}

\title[On a semitopological  semigroup $\boldsymbol{B}_{\omega}^{\mathscr{F}}$  ...]{On a semitopological  semigroup $\boldsymbol{B}_{\omega}^{\mathscr{F}}$  when a family $\mathscr{F}$ consists of inductive non-empty subsets of~$\omega$}

\author[Oleg Gutik and Mykola Mykhalenych]{Oleg Gutik and Mykola Mykhalenych}
\address{Faculty of Mechanics and Mathematics,
Ivan Franko National University of Lviv, Universytetska 1, Lviv, 79000, Ukraine}
\email{oleg.gutik@lnu.edu.ua,
ogutik@gmail.com, myhalenychmc@gmail.com}



\begin{abstract}
Let $\boldsymbol{B}_{\omega}^{\mathscr{F}}$ be the bicyclic semigroup extension for the family $\mathscr{F}$ of ${\omega}$-closed subsets of $\omega$ which is introduced in \cite{Gutik-Mykhalenych=2020}.
We study topologizations of the semigroup $\boldsymbol{B}_{\omega}^{\mathscr{F}}$ for the family $\mathscr{F}$ of inductive ${\omega}$-closed subsets of $\omega$. We generalize Eberhart-Selden and Bertman-West results about topologizations of the bicyclic semigroup \cite{Bertman-West-1976, Eberhart-Selden=1969} and show that every Hausdorff shift-continuous topology on the semigroup $\boldsymbol{B}_{\omega}^{\mathscr{F}}$ is discrete and if a Hausdorff semitopological semigroup $S$ contains $\boldsymbol{B}_{\omega}^{\mathscr{F}}$ as a proper dense subsemigroup then $S\setminus\boldsymbol{B}_{\omega}^{\mathscr{F}}$ is an ideal of $S$. Also, we prove the following dichotomy: every Hausdorff locally compact shift-continuous topology on $\boldsymbol{B}_{\omega}^{\mathscr{F}}$ with an adjoined zero is either compact or discrete. As a consequence of the last result we obtain that every Hausdorff locally compact semigroup topology on $\boldsymbol{B}_{\omega}^{\mathscr{F}}$ with an adjoined zero is discrete and every Hausdorff locally compact shift-continuous topology on the semigroup $\boldsymbol{B}_{\omega}^{\mathscr{F}}\sqcup I$ with an adjoined compact ideal $I$ is either compact or the ideal $I$ is open, which extent many results about locally compact topologizations of some classes of semigroups onto extensions of the semigroup $\boldsymbol{B}_{\omega}^{\mathscr{F}}$.
\end{abstract}

\maketitle


\section{Introduction}\label{section-1}

We shall follow the terminology of~\cite{Carruth-Hildebrant-Koch-1983, Clifford-Preston-1961, Clifford-Preston-1967, Engelking-1989, Ruppert-1984}. By $\omega$ we denote the set of all non-negative integers.

Let $\mathscr{P}(\omega)$ be  the family of all subsets of $\omega$. For any $F\in\mathscr{P}(\omega)$ and integer $n$ we put
\begin{equation*}
nF=\{n+k\colon k\in F\}.
\end{equation*}
A family $\mathscr{F}\subseteq\mathscr{P}(\omega)$ is called \emph{${\omega}$-closed} if $F_1\cap(-n+F_2)\in\mathscr{F}$ for all $n\in\omega$ and $F_1,F_2\in\mathscr{F}$.

A semigroup $S$ is called {\it inverse} if for any
element $x\in S$ there exists a unique $x^{-1}\in S$ such that
$xx^{-1}x=x$ and $x^{-1}xx^{-1}=x^{-1}$. The element $x^{-1}$ is called the {\it inverse of} $x\in S$.


If $S$ is an inverse semigroup then the semigroup operation on $S$ determines the following partial order $\preccurlyeq$
on $S$: $s\preccurlyeq t$ iff there exists $e\in E(S)$ such that $s=te$, for $s,t\in S$. This order is
called the {\em natural partial order} on $S$ \cite{Wagner-1952}.

The bicyclic monoid ${\mathscr{C}}(p,q)$ is the semigroup with the identity $1$ generated by two elements $p$ and $q$ and the condition $pq=1$. Thus each element of ${\mathscr{C}}(p,q)$) equals $q^mp^n$ for some $m,n\in\omega$ and the semigroup operation on ${\mathscr{C}}(p,q)$ can be described as follows
\begin{equation*}
    q^kp^l\cdot q^mp^n=q^{k+m-\min\{l,m\}}p^{l+n-\min\{l,m\}},
\end{equation*}
for each $k,l,m,n\in\omega$.
It is well known that the bicyclic monoid ${\mathscr{C}}(p,q)$ is a bisimple (and hence simple) combinatorial $E$-unitary inverse semigroup and every non-trivial congruence on ${\mathscr{C}}(p,q)$ is a group congruence \cite{Clifford-Preston-1961}.

On the other hand, we can we define the semigroup operation ``$\cdot$'' on the set $\boldsymbol{B}_{\omega}=\omega\times\omega$  in the following way
\begin{equation*}
  (i_1,j_1)\cdot(i_2,j_2)=
  \left\{
    \begin{array}{ll}
      (i_1-j_1+i_2,j_2), & \hbox{if~} j_1\leqslant i_2;\\
      (i_1,j_1-i_2+j_2), & \hbox{if~} j_1\geqslant i_2,
    \end{array}
  \right.
\end{equation*}
for each $i_1,i_2,j_1,j_2\in\omega$.
It is well-known that the semigroup $\boldsymbol{B}_{\omega}$ is isomorphic to the bicyclic monoid by the mapping $\mathfrak{h}\colon \mathscr{C}(p,q)\to \boldsymbol{B}_{\omega}$, $q^kp^l\mapsto (k,l)$ (see: \cite[Section~1.12]{Clifford-Preston-1961} or \cite[Exercise IV.1.11$(ii)$]{Petrich-1984}).

A {\it topological} ({\it semitopological}) {\it semigroup} is a topological space endowed with a continuous (separately continuous) semigroup operation. If $S$ is a~semigroup and $\tau$ is a topology on $S$ such that $(S,\tau)$ is a topological semigroup, then we shall call $\tau$ a \emph{semigroup} \emph{topology} on $S$, and if $\tau$ is a topology on $S$ such that $(S,\tau)$ is a semitopological semigroup, then we shall call $\tau$ a \emph{shift-continuous} \emph{topology} on~$S$. 

The well-known A.~Weil Theorem states that \emph{every locally compact monothetic topological group $G$} (i.e., $G$ contains a cyclic dense subgroup) \emph{is either compact or discrete} (see \cite{Weil-1938}). A semitopological semigroup $S$ is called \emph{monothetic} if it contains a cyclic dense subsemigroup. Locally compact and compact monothetic topological semigroups were studied by Hewitt \cite{Hewitt-1956},  Hofmann \cite{Hofmann-1960}, Koch  \cite{Koch-1957}, Numakura \cite{Numakura-1952} and others (for more related information see the books \cite{Carruth-Hildebrant-Koch-1986} and \cite{Hofmann-Mostert-1966}). Koch in \cite{Koch-1957} posed the following problem: ``\emph{If $S$ is a locally compact monothetic semigroup and $S$ has an identity, must $S$ be compact?}'' From the other hand, Zelenyuk in \cite{Zelenyuk-1988} constructed a countable monothetic locally compact topological semigroup without an identity which is neither compact nor discrete and in \cite{Zelenyuk-2019} he constructed a monothetic locally compact topological monoid with the same property. The topological properties of monothetic locally compact (semi)topological semigroups are studied in \cite{Banakh-Bardyla-Guran-Gutik-Ravsky-2020, Guran-Kisil-2012, Zelenyuk-2020, Zelenyuk-Zelenyuk-2020}.

In the paper \cite{Gutik=2015} it is proved that every Hausdorff locally compact shift-continuous topology on the bicyclic monoid with an adjoined zero is either compact or discrete. This result was extended by Bardyla to the polycyclic monoid \cite{Bardyla-2016} and graph inverse semigroups \cite{Bardyla-2018}, and by Mokrytskyi to the monoid of order isomorphisms between principal filters of $\mathbb{N}^n$ with an adjoined zero \cite{Mokrytskyi-2019}. Also, in \cite{Gutik-Maksymyk-2019} it is proved that the extended bicyclic semigroup $\mathscr{C}_\mathscr{\mathbb{Z}}^0$ with an adjoined zero admits continuum many different  shift-continuous topologies, however every Hausdorff locally compact semigroup topology on $\mathscr{C}_\mathscr{\mathbb{Z}}^0$ is discrete. In \cite{Bardyla=2021??} Bardyla proved that a Hausdorff locally compact semitopological McAlister semigroup $\mathcal{M}_1$ is either compact or discrete. However, this dichotomy does not hold for the McAlister semigroup $\mathcal{M}_2$ and moreover, $\mathcal{M}_2$ admits continuum many different Hausdorff locally compact inverse semigroup topologies \cite{Bardyla=2021??}. Also, different locally compact semitopological semigroups with zero were studied in \cite{Gutik-2018, Gutik-KhylynskyiP=2022, Maksymyk-2019}.

Next we shall describe the construction which is introduced in \cite{Gutik-Mykhalenych=2020}.

Let  $\mathscr{F}$ be an ${\omega}$-closed subfamily of $\mathscr{P}(\omega)$. We can we define the semigroup operation ``$\cdot$'' on the set $\boldsymbol{B}_{\omega}\times\mathscr{F}$  in the following way
\begin{equation*}
  (i_1,j_1,F_1)\cdot(i_2,j_2,F_2)=
  \left\{
    \begin{array}{ll}
      (i_1-j_1+i_2,j_2,(j_1-i_2+F_1)\cap F_2), & \hbox{if~} j_1\leqslant i_2;\\
      (i_1,j_1-i_2+j_2,F_1\cap (i_2-j_1+F_2)), & \hbox{if~} j_1\geqslant i_2,
    \end{array}
  \right.
\end{equation*}
for each $i_1,i_2,j_1,j_2\in\omega$.
In \cite{Gutik-Mykhalenych=2020} is proved that $(\boldsymbol{B}_{\omega}\times\mathscr{F},\cdot)$ is an inverse semigroup. Moreover, if a family  $\mathscr{F}$ contains the empty set $\varnothing$ then the set
$ 
  \boldsymbol{I}=\{(i,j,\varnothing)\colon i,j\in\omega\}
$ 
is an ideal of the semigroup $(\boldsymbol{B}_{\omega}\times\mathscr{F},\cdot)$. For any ${\omega}$-closed family $\mathscr{F}\subseteq\mathscr{P}(\omega)$ the following semigroup
\begin{equation*}
  \boldsymbol{B}_{\omega}^{\mathscr{F}}=
\left\{
  \begin{array}{ll}
    (\boldsymbol{B}_{\omega}\times\mathscr{F},\cdot)/\boldsymbol{I}, & \hbox{if~} \varnothing\in\mathscr{F};\\
    (\boldsymbol{B}_{\omega}\times\mathscr{F},\cdot), & \hbox{if~} \varnothing\notin\mathscr{F}
  \end{array}
\right.
\end{equation*}
is defined in \cite{Gutik-Mykhalenych=2020}. The semigroup $\boldsymbol{B}_{\omega}^{\mathscr{F}}$ generalizes the bicyclic monoid and the countable semigroup of matrix units. In \cite{Gutik-Mykhalenych=2020} it is provd that $\boldsymbol{B}_{\omega}^{\mathscr{F}}$ is a combinatorial inverse semigroup and Green's relations, the natural partial order on $\boldsymbol{B}_{\omega}^{\mathscr{F}}$ and its set of idempotents are described. The criteria of simplicity, $0$-simplicity, bisimplicity, $0$-bisimplicity of the semigroup $\boldsymbol{B}_{\omega}^{\mathscr{F}}$ and when $\boldsymbol{B}_{\omega}^{\mathscr{F}}$ has the identity, is isomorphic to the bicyclic semigroup or the countable semigroup of matrix units are given. In particularly in \cite{Gutik-Mykhalenych=2020} is proved that the semigroup $\boldsymbol{B}_{\omega}^{\mathscr{F}}$ is isomorphic to the semigrpoup of ${\omega}{\times}{\omega}$-matrix units if and only if $\mathscr{F}$ consists of a singleton set and the empty set.

A subset $F$ of $\omega$ is called \emph{inductive}  if $n+1\in F$ provided $n\in F$. It is obvious that the empty set $\varnothing$ is inductive. Also, a set $F\subseteq \omega$ is inductive  if and only if $F\subseteq -1+F$ (see \cite[Lemma~6]{Gutik-Mykhalenych=2020}). If $\mathscr{F}$ is a family in $\mathscr{P}(\omega)$ of non-empty inductive subsets then the semigroup $\boldsymbol{B}_{\omega}^{\mathscr{F}}$ is a simple monoid (see \cite[Corollary~2 and Theorem~4]{Gutik-Mykhalenych=2020}). Group congruences on the semigroup  $\boldsymbol{B}_{\omega}^{\mathscr{F}}$ and its homomorphic retracts when the family $\mathscr{F}$  consists of inductive non-empty subsets of $\omega$ are studied in \cite{Gutik-Mykhalenych=2021}. It is proven that a congruence $\mathfrak{C}$ on $\boldsymbol{B}_{\omega}^{\mathscr{F}}$ is a group congruence if and only if its restriction on a subsemigroup of $\boldsymbol{B}_{\omega}^{\mathscr{F}}$, which is isomorphic to the bicyclic semigroup, is not the identity relation. Also, all non-trivial homomorphic retracts of the semigroup $\boldsymbol{B}_{\omega}^{\mathscr{F}}$ are described.

By Proposition~1 of \cite{Gutik-Mykhalenych=2021} for any $\omega$-closed family $\mathscr{F}$ of inductive subsets in $\mathscr{P}(\omega)$ there exists an $\omega$-closed family $\mathscr{F}^*$ of inductive subsets in $\mathscr{P}(\omega)$ such that $[0)\in \mathscr{F}^*$ and the semigroups $\boldsymbol{B}_{\omega}^{\mathscr{F}}$ and $\boldsymbol{B}_{\omega}^{\mathscr{F}^*}$ are isomorphic. Hence without loss of generality we may assume that the family $\mathscr{F}$ contains the set $[0)$. Also, $\omega$-closeness of $\mathscr{F}$ implies that if $[k)\in\mathscr{F}$ for some $k\in\omega$ then $[l)\in\mathscr{F}$ for all $l\leqslant k$ with $l\in\omega$.

In this paper we extend the results of \cite{Bertman-West-1976, Eberhart-Selden=1969, Gutik=2015} onto the semigroup $\boldsymbol{B}_{\omega}^{\mathscr{F}}$ for an $\omega$-closed family $\mathscr{F}$ of inductive nonempty subsets of $\omega$. In particular we show that every Hausdorff shift-continuous topology  on the semigroup $\boldsymbol{B}_{\omega}^{\mathscr{F}}$ is discrete and if a Hausdorff semitopological semigroup $S$ contains $\boldsymbol{B}_{\omega}^{\mathscr{F}}$ as a proper dense subsemigroup then $S\setminus\boldsymbol{B}_{\omega}^{\mathscr{F}}$ is an ideal of $S$. Also, we prove that every Hausdorff locally compact shift-continuous topology on $\boldsymbol{B}_{\omega}^{\mathscr{F}}$ with an adjoined zero is either compact or discrete.


\section{On a topologization and a closure of the monoid $\boldsymbol{B}_{\omega}^{\mathscr{F}}$}\label{section-2}

Later we shall need the following proposition from \cite{Gutik-Mykhalenych=2020}, which describes the natural partial order on the semigroup $\boldsymbol{B}_{\omega}^{\mathscr{F}}$ in the general case of $\mathscr{F}$.

\begin{proposition}[\!\!{\cite[Proposition~6]{Gutik-Mykhalenych=2020}}]\label{proposition-2.1}
Let $(i_1,j_1,F_1)$ and $(i_2,j_2,F_2)$ be non-zero elements of the semigroup $\boldsymbol{B}_{\omega}^{\mathscr{F}}$. Then  $(i_1,j_1,F_1)\preccurlyeq(i_2,j_2,F_2)$ if and only if $F_1\subseteq -k+F_2$ and $i_1-i_2=j_1-j_2=k$ for some $k\in\omega$.
\end{proposition}

\begin{proposition}\label{proposition-2.2}
For every non-zero elements $(i_1,j_1,F_1)$ and $(i_2,j_2,F_2)$ of the semigroup $\boldsymbol{B}_{\omega}^{\mathscr{F}}$, both sets
\begin{equation*}
\left\{(i,j,F)\in\boldsymbol{B}_{\omega}^{\mathscr{F}}\colon (i_1,j_1,F_1)\cdot(i,j,F)=(i_2,j_2,F_2)\right\}
\end{equation*}
and
\begin{equation*}
\left\{(i,j,F)\in\boldsymbol{B}_{\omega}^{\mathscr{F}}\colon(i,j,F)\cdot(i_1,j_1,F_1)=(i_2,j_2,F_2)\right\}
\end{equation*}
are finite.
\end{proposition}

\begin{proof}
It is obvious that the set
\begin{equation*}
A=\left\{(i,j,F)\in\boldsymbol{B}_{\omega}^{\mathscr{F}}\colon (i_1,j_1,F_1)\cdot(i,j,F)=(i_2,j_2,F_2)\right\}
\end{equation*}
is a subset of
\begin{equation*}
B=\left\{(i,j,F)\in\boldsymbol{B}_{\omega}^{\mathscr{F}}\colon (i_1,j_1,F_1)^{-1}\cdot(i_1,j_1,F_1)\cdot(i,j,F)=(i_1,j_1,F_1)^{-1}\cdot(i_2,j_2,F_2)\right\}.
\end{equation*}
By Lemmas~2 and~3 from \cite{Gutik-Mykhalenych=2020},  $(i_1,j_1,F_1)^{-1}\cdot(i_1,j_1,F_1)=(j_1,j_1,F_1)$ is an idempotent of $\boldsymbol{B}_{\omega}^{\mathscr{F}}$, and hence Lemma~1.4.6 of \cite{Lawson-1998} implies that
\begin{equation*}
B=\left\{(i,j,F)\in\boldsymbol{B}_{\omega}^{\mathscr{F}}\colon (i_1,j_1,F_1)^{-1}\cdot(i_2,j_2,F_2)\preccurlyeq (i,j,F)\right\}.
\end{equation*}
Proposition~\ref{proposition-2.1} implies that there exist finitely many $i,j\in\omega$ and $F\in\mathscr{F}$ such that $(i_1,j_1,F_1)^{-1}\cdot(i_2,j_2,F_2)\preccurlyeq (i,j,F)$, and hence the set $A$ is finite.
\end{proof}

The following theorem generalizes the results on the topologizabily of the bicyclic monoid obtained in \cite{Eberhart-Selden=1969} and \cite{Bertman-West-1976}.

\begin{theorem}\label{theorem-2.3}
Let $\mathscr{F}$ be a family of non-empty inductive subsets of $\omega$. Then every Hausdorff shift-continuous topology $\tau$ on the semigroup $\boldsymbol{B}_{\omega}^{\mathscr{F}}$ is discrete.
\end{theorem}

\begin{proof}
For any element $(i,j,[k))$ of the semigroup $\boldsymbol{B}_{\omega}^{\mathscr{F}}$ we have that
\begin{align*}
  (0,0,[1))\cdot(i,j,[k))&=(i,j,(-i+[1))\cap[k))= \\
  &=
  \left\{
    \begin{array}{ll}
      (0,j,[1)), & \hbox{if~~} i=0 \hbox{~~and~~} k\leqslant 1;\\
      (0,j,[k)), & \hbox{if~~} i=0 \hbox{~~and~~} k\geqslant 2; \\
      (i,j,[k)), & \hbox{if~~} i\geqslant 1
    \end{array}
  \right.
\end{align*}
and
\begin{align*}
  (i,j,[k))\cdot(0,0,[1))&=(i,j,[k)\cap(-j+[1)))= \\
  &=
  \left\{
    \begin{array}{ll}
      (i,0,[1)), & \hbox{if~~} j=0 \hbox{~~and~~} k\leqslant 1;\\
      (i,0,[k)), & \hbox{if~~} j=0 \hbox{~~and~~} k\geqslant 2; \\
      (i,j,[k)), & \hbox{if~~} j\geqslant 1.
    \end{array}
  \right.
\end{align*}
So
\begin{equation}\label{eq-2.1}
  \left((0,0,[1))\cdot\boldsymbol{B}_{\omega}^{\mathscr{F}}\right)\cup\left(\boldsymbol{B}_{\omega}^{\mathscr{F}}\cdot(0,0,[1))\right)= \boldsymbol{B}_{\omega}^{\mathscr{F}}\setminus\{(0,0,[0))\}.
\end{equation}
Since $\tau$ is Hausdorff, every retract of $\big(\boldsymbol{B}_{\omega}^{\mathscr{F}},\tau\big)$ is its closed subset. It is obvious  that $(0,0,[1))\cdot\boldsymbol{B}_{\omega}^{\mathscr{F}}$ and $\boldsymbol{B}_{\omega}^{\mathscr{F}}\cdot(0,0,[1))$ are retracts of the topological space $\big(\boldsymbol{B}_{\omega}^{\mathscr{F}},\tau\big)$, because $(0,0,[1))$ is an idempotent of the semigroup $\boldsymbol{B}_{\omega}^{\mathscr{F}}$. Since any retract of Hausdorff space is closed (see: \cite[Ex.~1.5.c]{Engelking-1989}), equality \eqref{eq-2.1} implies that the point $(0,0,[0))$ is an isolated point of the space $\big(\boldsymbol{B}_{\omega}^{\mathscr{F}},\tau\big)$. By Corollary~2 of \cite{Gutik-Mykhalenych=2020}, $\boldsymbol{B}_{\omega}^{\mathscr{F}}$ is a simple semigroup. This implies that for any
$(i,j,[k))\in\boldsymbol{B}_{\omega}^{\mathscr{F}}$ there exist $(i_1,j_1,[k_1)),(i_2,j_2,[k_2))\in\boldsymbol{B}_{\omega}^{\mathscr{F}}$
such that $(i_1,j_1,[k_1))\cdot(i,j,[k))\cdot(i_2,j_2,[k_2))=(0,0,[0))$, and moreover by Proposition~\ref{proposition-2.2} the equation $(i_1,j_1,[k_1))\cdot\chi\cdot(i_2,j_2,[k_2))=(0,0,[0))$ has finitely many solutions in the semigroup $\boldsymbol{B}_{\omega}^{\mathscr{F}}$.
Since $(0,0,[0))$ is an isolated point of $\big(\boldsymbol{B}_{\omega}^{\mathscr{F}},\tau\big)$, the separate continuity of the semigroup operation in $\big(\boldsymbol{B}_{\omega}^{\mathscr{F}},\tau\big)$ and the above arguments imply that $\big(\boldsymbol{B}_{\omega}^{\mathscr{F}},\tau\big)$ is the discrete space.
\end{proof}

The following proposition generalizes the results obtained for the bicyclic monoid in \cite{Eberhart-Selden=1969} and \cite{Gutik=2015}.

\begin{proposition}\label{proposition-2.4}
Let $\mathscr{F}$ be a family of non-empty inductive subsets of $\omega$ and $\boldsymbol{B}_{\omega}^{\mathscr{F}}$ be a proper dense subsemigroup of a Hausdorff semitopological semigroup $S$. Then $I=S\setminus\boldsymbol{B}_{\omega}^{\mathscr{F}}$ is a closed ideal of $S$.
\end{proposition}

\begin{proof}
By Theorem~\ref{theorem-2.3}, $\boldsymbol{B}_{\omega}^{\mathscr{F}}$ is a dense discrete subspace of $S$, and hence  $\boldsymbol{B}_{\omega}^{\mathscr{F}}$ is an open subspace of $S$.

Fix an arbitrary element $y\in I$. If $xy = z\notin I$ for some $x\in \boldsymbol{B}_{\omega}^{\mathscr{F}}$ then there exists an open neighbourhood $U(y)$ of the point $y$ in the space $S$ such that $\{x\}\cdot U(y) =\{z\}\subset\boldsymbol{B}_{\omega}^{\mathscr{F}}$. The neighbourhood $U(y)$ contains infinitely many elements of the semigroup $\boldsymbol{B}_{\omega}^{\mathscr{F}}$, which contradicts Proposition~\ref{proposition-2.2}. The obtained contradiction implies that  $xy = z\in I$ for all $x\in \boldsymbol{B}_{\omega}^{\mathscr{F}}$ and $y\in I$. The proof of the statement that $yx\in I$ for all $x\in \boldsymbol{B}_{\omega}^{\mathscr{F}}$ and $y\in I$ is similar.

Suppose to the contrary that $xy = z\notin I$  for some $x,y\in I$. Then $z\in \boldsymbol{B}_{\omega}^{\mathscr{F}}$ and the separate continuity of the semigroup operation in $S$ implies that there exist open neighbourhoods $U(x)$ and $U(y)$ of the points $x$ and y in $S$, respectively,
such that $\{x\}\cdot U(y)=\{z\}$ and $U(x)\cdot \{y\}=\{z\}$. Since both neighbourhoods $U(x)$ and $U(y)$ contain infinitely many elements of the semigroup $\boldsymbol{B}_{\omega}^{\mathscr{F}}$, any of equalities $\{x\}\cdot U(y)=\{z\}$ and $U(x)\cdot \{y\}=\{z\}$ contradicts mentioned above
Proposition~\ref{proposition-2.2}. The obtained contradiction implies that $xy\in I$.
\end{proof}

\section{On a semitopological locally compact monoid $\boldsymbol{B}_{\omega}^{\mathscr{F}}$ with an adjoined zero}\label{section-3}

In this section we assume that $S=\boldsymbol{B}_{\omega}^{\mathscr{F}}\sqcup\{\boldsymbol{0}\}$, i.e., $S$ the semigroup $\boldsymbol{B}_{\omega}^{\mathscr{F}}$ with an adjoined zero $\boldsymbol{0}$. We observe that the semigroup $S$ is isomorphic to the semigroup $\boldsymbol{B}_{\omega}^{\mathscr{F}^0}$, where the family $\mathscr{F}^0$ consists of elements of $\mathscr{F}$ and the empty set $\varnothing$ (see \cite[Lemma~1]{Gutik-Mykhalenych=2020}).
 Later in the following series of lemmas we assume that $S$ is a Hausdorff locally compact semitopological semigroup with the nonisolated zero $\boldsymbol{0}$ and  the family $\mathscr{F}$ is $\omega$-closed and consists of nonempty inductive subsets of $\omega$.

By Theorem~\ref{theorem-2.3}, $\boldsymbol{B}_{\omega}^{\mathscr{F}}$ is a discrete subspace of $S$. This implies the following lemma.

\begin{lemma}\label{lemma-3.1}
Let $U(\boldsymbol{0})$ and $V(\boldsymbol{0})$ be any compact-and-open neighbourhoods of $\boldsymbol{0}$ in $S$. Then the set $U(\boldsymbol{0})\setminus V(\boldsymbol{0})$ is finite.
\end{lemma}

Since $\boldsymbol{B}_{\omega}^{\mathscr{F}}$ is a discrete subspace of $S$ without loss of generality we consider only compact-and-open neighbourhoods of zero in $S$.

In the general case if the family $\mathscr{F}$ contains an inductive set $F$ then
\begin{equation*}
  \boldsymbol{B}_{\omega}^{\{F\}}=\{(i,j,F)\colon i,j\in\omega\}
\end{equation*}
is an inverse subsemigroup of $\boldsymbol{B}_{\omega}^{\mathscr{F}}$. Moreover, if $F$ is non-empty then by Proposition~3 of \cite{Gutik-Mykhalenych=2020}, $\boldsymbol{B}_{\omega}^{\{F\}}$ is isomorphic to the bicyclic semigroup.

\begin{lemma}\label{lemma-3.2}
For any neighbourhood $U(\boldsymbol{0})$  of $\boldsymbol{0}$ in $S$ there exists $F\in \mathscr{F}$ such that the set $U(\boldsymbol{0})\cap \boldsymbol{B}_{\omega}^{\{F\}}$ is infinite.
\end{lemma}

\begin{proof}
The statement of the lemma is obvious when the family $\mathscr{F}$ is finite. Hence we assume that $\mathscr{F}$ is infinite.

Suppose to the contrary that the set $U(\boldsymbol{0})\cap \boldsymbol{B}_{\omega}^{\{F\}}$ is finite for any $F\in\mathscr{F}$. By the separate continuity of the semigroup operation in $S$ there exists a neighbourhood $V(\boldsymbol{0})\subseteq U(\boldsymbol{0})$ of $\boldsymbol{0}$ in $S$ such that $V(\boldsymbol{0})\cdot(0,1,[0))\subseteq U(\boldsymbol{0})$. Since $(i,j,F)\cdot(0,1,[0))=(i,j+1,F)$ for all $i,j\in\omega$ and any $F\in\mathscr{F}$, we obtain that $U(\boldsymbol{0})\setminus V(\boldsymbol{0})$ is an infinite set, which contradicts Lemma~\ref{lemma-3.1}. The obtained contradiction implies the statement of the lemma.
\end{proof}

\begin{lemma}\label{lemma-3.3}
For any neighbourhood $U(\boldsymbol{0})$  of $\boldsymbol{0}$ in $S$ there exists $F\in \mathscr{F}$ such that the set $\boldsymbol{B}_{\omega}^{\{F\}}\setminus U(\boldsymbol{0})$ is finite.
\end{lemma}

\begin{proof}
By Lemma~\ref{lemma-3.2} there exists $F\in \mathscr{F}$ such that the set $U(\boldsymbol{0})\cap \boldsymbol{B}_{\omega}^{\{F\}}$ is infinite. By Theorem~\ref{theorem-2.3} all non-zero elements of the semigroup $S$ are isolated points in $S$, and hence $\boldsymbol{B}_{\omega}^{\{F\}}\cup\{\boldsymbol{0}\}$ is a closed subset of $S$, which by Corollary~3.3.10 of \cite{Engelking-1989} is locally compact. It obvious that $\boldsymbol{B}_{\omega}^{\{F\}}\cup\{\boldsymbol{0}\}$ is a subsemigroup of $S$, which by Proposition~3 of \cite{Gutik-Mykhalenych=2020} is algebraically isomorphic to the bicyclic monoid with an adjoined zero. By Theorem~1 of \cite{Gutik=2015}, $\boldsymbol{B}_{\omega}^{\{F\}}\cup\{\boldsymbol{0}\}$ is compact, which implies the statement of the lemma.
\end{proof}

\begin{lemma}\label{lemma-3.4}
For any neighbourhood $U(\boldsymbol{0})$  of $\boldsymbol{0}$ in $S$ and any $F\in \mathscr{F}$ the set $\boldsymbol{B}_{\omega}^{\{F\}}\setminus U(\boldsymbol{0})$ is finite.
\end{lemma}

\begin{proof}
If the family $\mathscr{F}$ is a singleton then the statement of the lemma follows from Theorem~1 of \cite{Gutik=2015}. Hence we assume that $\mathscr{F}$ is not a singleton.

We shall prove the statement of the lemma by induction.

By Lemma~\ref{lemma-3.3} there exists $F_0\in \mathscr{F}$ such that the set $\boldsymbol{B}_{\omega}^{\{F_0\}}\setminus U(\boldsymbol{0})$ is finite. Since $F_0$ is inductive there exists $k_0\in \omega$ such that $F=[k_0)$. This proves that the base of induction holds.

Next we shall show the inductive step. We consider two cases:
\begin{enumerate}
  \item[(1)] if $[k),[k+1)\in\mathscr{F}$ then  the statement that the set $\boldsymbol{B}_{\omega}^{\{[k)\}}\setminus U(\boldsymbol{0})$ is finite implies that the set $\boldsymbol{B}_{\omega}^{\{[k+1)\}}\setminus U(\boldsymbol{0})$ is finite, too;
  \item[(2)] if $[k),[k-1)\in\mathscr{F}$ then  the statement that the set $\boldsymbol{B}_{\omega}^{\{[k)\}}\setminus U(\boldsymbol{0})$ is finite implies that the set $\boldsymbol{B}_{\omega}^{\{[k-1)\}}\setminus U(\boldsymbol{0})$ is finite, too.
\end{enumerate}

(1) The neighbourhood $U(\boldsymbol{0})$ contains almost all elements of the form $(i,j,[k))$, $i,j\in\omega$. The separate continuity of the semigroup operation in $S$ implies that there exists a neighbourhood $V(\boldsymbol{0})\subseteq U(\boldsymbol{0})$ of $\boldsymbol{0}$ in $S$ such that $(1,1,[k+1))\cdot V(\boldsymbol{0})\subseteq U(\boldsymbol{0})$ and the set $U(\boldsymbol{0})\setminus V(\boldsymbol{0})$ is finite. This implies that $V(\boldsymbol{0})$ contains almost all elements of the semigroup $\boldsymbol{B}_{\omega}^{\{[k)\}}$.
Then the equalities
\begin{equation*}
  (1,1,[k+1))\cdot(0,p,[k))=(1,1+p,[k+1)\cap(-1+[k)))=(1,1+p,[k+1)), \qquad p\in\omega,
\end{equation*}
implies that the neighbourhood $U(\boldsymbol{0})$ contains infinitely many elements of the semigroup $\boldsymbol{B}_{\omega}^{\{[k+1)\}}$. By Corollary~3.3.10 of \cite{Engelking-1989} and Proposition~3 of \cite{Gutik-Mykhalenych=2020}, $\boldsymbol{B}_{\omega}^{\{[k)\}}\cup\{\boldsymbol{0}\}$ is a locally compact semitopological semigroup which is algebraically isomorphic to the bicyclic monoid with an adjoined zero. By Theorem~1 of \cite{Gutik=2015} the set $\boldsymbol{B}_{\omega}^{\{[k+1)\}}\setminus U(\boldsymbol{0})$ is finite.

(2) Since the neighbourhood $U(\boldsymbol{0})$ contains almost all elements of the form $(i,j,[k))$, $i,j\in\omega$, the separate continuity of the semigroup operation in $S$ implies that there exists a neighbourhood $V(\boldsymbol{0})\subseteq U(\boldsymbol{0})$ of $\boldsymbol{0}$ in $S$ such that $(1,1,[0))\cdot V(\boldsymbol{0})\subseteq U(\boldsymbol{0})$ and the set $U(\boldsymbol{0})\setminus V(\boldsymbol{0})$ is finite. Since $V(\boldsymbol{0})$ contains almost all elements of the semigroup $\boldsymbol{B}_{\omega}^{\{[k)\}}$ and the set $U(\boldsymbol{0})\setminus V(\boldsymbol{0})$ is finite, the equalities
\begin{equation*}
  (1,1,[0))\cdot (0,p,[k))=(1,1+p,[0)\cap(-1+[k))=(1,1+p,[k-1)), \qquad p\in\omega,
\end{equation*}
imply that the neighbourhood $U(\boldsymbol{0})$ contains infinitely many elements of the semigroup $\boldsymbol{B}_{\omega}^{\{[k-1)\}}$. By Corollary~3.3.10 of \cite{Engelking-1989} and Proposition~3 of \cite{Gutik-Mykhalenych=2020}, $\boldsymbol{B}_{\omega}^{\{[k-1)\}}\cup\{\boldsymbol{0}\}$ is a locally compact semitopological semigroup which is algebraically isomorphic to the bicyclic monoid with an adjoined zero. By Theorem~1 of \cite{Gutik=2015} the set $\boldsymbol{B}_{\omega}^{\{[k_0-1)\}}\setminus U(\boldsymbol{0})$ is finite.
\end{proof}

\begin{lemma}\label{lemma-3.5}
For any neighbourhood $U(\boldsymbol{0})$  of $\boldsymbol{0}$ in $S$ the set $S\setminus U(\boldsymbol{0})$ is finite.
\end{lemma}

\begin{proof}
In the case when the family $\mathscr{F}$ is finite the statement of the lemma follows from Lemma~\ref{lemma-3.4}, and hence later we assume that $\mathscr{F}$ is infinite.

Suppose to the contrary that there exists a neighbourhood $U(\boldsymbol{0})$  of $\boldsymbol{0}$ in $S$ such that the set $S\setminus U(\boldsymbol{0})$ is infinite. By Lemma~\ref{lemma-3.4} there exists a sequence
\begin{equation*}
  \left\{(m_i,n_i,[k_i))\right\}_{i\in\omega}\subseteq S\setminus \{\boldsymbol{0}\}
\end{equation*}
such that $k_i=k_j$ if and only if $i=j$ and $(m_i,n_i,[k_i))\notin U(\boldsymbol{0})$ and $(m_i+1,n_i+1,[k_i))\in U(\boldsymbol{0})$ for all $i\in\omega$. By the separate continuity of the semigroup operation in $S$ there exists a neighbourhood $V(\boldsymbol{0})\subseteq U(\boldsymbol{0})$ of $\boldsymbol{0}$ in $S$ such that
\begin{equation*}
  (0,1,[0))\cdot V(\boldsymbol{0})\cdot(1,0,[0))\subseteq U(\boldsymbol{0}).
\end{equation*}
Then we have that
\begin{align*}
  (0,1,[0))\cdot (m_i+1,n_i+1,[k_i))\cdot(1,0,[0))&=(m_i,n_i+1,(-1+[0))\cap[k_i))\cdot(1,0,[0))= \\
  &=(m_i,n_i+1,[k_i))\cdot(1,0,[0))=\\
  &=(m_i,n_i,[k_i)\cap(-1+[0)))=\\
  &=(m_i,n_i,[k_i)),
\end{align*}
which contradicts that $U(\boldsymbol{0})\setminus V(\boldsymbol{0})$ is an infinite set, a contradiction. The obtained contradiction implies the statement of the lemma.
\end{proof}

\begin{definition}[\!\cite{Chuchman-Gutik-2010}]\label{definitio-3.6}
We shall say that a semigroup $S$ has
the $\textsf{F}$-\emph{property} if for every $a,b,c,d\in S^1$ the sets $\{x\in S\mid a\cdot x=b\}$ and $\{x\in S\mid x\cdot
     c=d\}$ are finite.
\end{definition}

Lemma~\ref{lemma-3.7} was proved in \cite{Gutik-KhylynskyiP=2022} and it shows that on the semigroup $T$ with the $\textsf{F}$-property there exists a Hausdorff compact shift-continuous  topology $\tau_{\operatorname{\textsf{Ac}}}$.

\begin{lemma}[\!\!{\cite{Gutik-KhylynskyiP=2022}}]\label{lemma-3.7}
Let $T$ be a semigroup with the $\textsf{F}$-property and $T^0$ be the semigroup $T$ with an adjoined zero. Let $\tau_{\operatorname{\textsf{Ac}}}$ be the topology on $T^0$  such that
\begin{itemize}
  \item[$(i)$] every element of $T$ is an isolated point in the space $(T^{0},\tau_{\operatorname{\textsf{Ac}}})$;
  \item[$(ii)$] the family $\mathscr{B}({0})=\left\{U\subseteq T^{{0}}\colon U\ni {0} \hbox{~and~} T^{{0}}\setminus U \hbox{~is finite}\right\}$ determines a base of the topology $\tau_{\operatorname{\textsf{Ac}}}$ at zero ${0}\in T^{{0}}$.
\end{itemize}
Then $(T^{0},\tau_{\operatorname{\textsf{Ac}}})$ is a Hausdorff compact semitopological  semigroup.
\end{lemma}

\begin{remark}\label{remark-3.8}
By Theorem~\ref{theorem-2.3} the discrete topology is a unique Hausdorff  shift-continuous topology on a semigroup $T$. So $\tau_{\operatorname{\textsf{Ac}}}$ is the unique compact shift-continuous topology on $T$.
\end{remark}

\begin{theorem}\label{theorem-3.9}
Let $\mathscr{F}$ be a family of inductive non-empty subsets of $\omega$ and $S$ be the semigroup  $\boldsymbol{B}_{\omega}^{\mathscr{F}}$  with an adjoined zero. Then every Hausdorff locally compact shift-continuous topology on $S$ is either compact or discrete.
\end{theorem}

\begin{proof}
In the case when zero of $S$ is an isolated point of $S$ the statement of the theorem follows from Theorem~\ref{theorem-2.3}. If zero of $S$ is a non-isolated point of $S$ then we apply Lemma~\ref{lemma-3.5}.
\end{proof}

Since the bicyclic monoid embeds into no Hausdorff compact topological semigroup \cite{Anderson-Hunter-Koch-1965} and by Proposition~3 of \cite{Gutik-Mykhalenych=2020} the semigroup  $\boldsymbol{B}_{\omega}^{\mathscr{F}}$ contains an isomorphic copy of the bicyclic monoid, Theorem~\ref{theorem-3.9} implies the following theorem.

\begin{theorem}\label{theorem-3.10}
Let $\mathscr{F}$ be a  family of inductive non-empty subsets of $\omega$ and $S$ be the semigroup  $\boldsymbol{B}_{\omega}^{\mathscr{F}}$  with an adjoined zero. Then every Hausdorff locally compact semigroup topology on $S$ is  discrete.
\end{theorem}

\begin{remark}\label{remark-3.11}
On the other hand, in \cite{Gutik=2015} is constructed the \v{C}ech-complete non-discrete metrizable semigroup topology on the bicyclic semigroup with the adjoined zero.
\end{remark}

We need the following simple lemma, which is implied from separate continuity of the semigroup operation in semitopological semigroups.

\begin{lemma}\label{lemma-3.12}
Let $X$ be a Hausdorff semitopological semigroup and $I$ be a compact ideal in $X$. Then the Rees-quotient semigroup $X/I$ with the quotient topology is a Hausdorff semitopological semigroup.
\end{lemma}

The proof of the following lemma is simple (see \cite{Gutik-KhylynskyiP=2022}).

\begin{lemma}\label{lemma-3.13}
Let $X$ be a Hausdorff locally compact space and $I$ be a compact subset of $X$. Then there exists an open neighbourhood $U(I)$ of $I$ with the compact closure $\overline{U(I)}$.
\end{lemma}

\begin{theorem}\label{theorem-3.14}
Let $\mathscr{F}$ be a  family of inductive non-empty subsets of $\omega$.
Let $(S_I,\tau)$ be a Hausdorff locally compact semitopological semigroup, where $S_I=\boldsymbol{B}_{\omega}^{\mathscr{F}}\sqcup I$ and $I$ is a compact ideal of $S_I$. Then either $(S_I,\tau)$ is a compact semitopological semigroup or the ideal $I$ is open.
\end{theorem}

\begin{proof}
Suppose that $I$ is not open. By Lemma~\ref{lemma-3.12} the Rees-quotient semigroup $S_I/I$ with the quotient topology $\tau_{\operatorname{\textsf{q}}}$ is a semitopological semigroup. Let $\pi\colon S_I\to S_I/I$ be the natural homomorphism which is a quotient map.
Since the Rees-quotient semigroup $S_I/I$ is naturally isomorphic to the semigroup $S$, without loss of generality we can assume that $\pi(S_I)=S$ and the image $\pi(I)$ is the zero $\boldsymbol{0}$ of $S$.

By Lemma~\ref{lemma-3.13} there exists an open neighbourhood $U(I)$ of $I$ with the compact closure $\overline{U(I)}$. Since by Theorem~\ref{theorem-2.3} every point of $\boldsymbol{B}_{\omega}^{\mathscr{F}}$ is isolated in $(S_I,\tau)$ we have that $\overline{U(I)}=U(I)$ and its image $\pi(U(I))$ is compact-and-open neighbourhood of zero in $S$. Since for any open neighbourhood $V(I)$ of $I$ in $(S_I,\tau)$ the set $\overline{U(I)\cap V(I)}$ is compact, Theorem~\ref{theorem-3.9} implies that $S\setminus\pi(U(I))$ is finite for any compact-and-open neighbourhood $U(I)$ of $I$ in $(S_I,\tau)$. Then compactness of $I$ implies that $(S_I,\tau)$ is compact as well.
\end{proof}

Since the bicyclic monoid embeds into no Hausdorff compact topological semigroup \cite{Anderson-Hunter-Koch-1965} and by Proposition~3 of \cite{Gutik-Mykhalenych=2020} the semigroup  $\boldsymbol{B}_{\omega}^{\mathscr{F}}$ contains an isomorphic copy of the bicyclic monoid, Theorem~\ref{theorem-3.14} implies

\begin{theorem}\label{theorem-3.15}
Let $\mathscr{F}$ be an $\omega$-closed family of inductive non-empty subsets of $\omega$.
Let $(S_I,\tau)$ be a Hausdorff locally compact topological semigroup, where $S_I=\boldsymbol{B}_{\omega}^{\mathscr{F}}\sqcup I$ and $I$ is a compact ideal of $S_I$. Then the ideal $I$ is open.
\end{theorem}

\section*{Acknowledgements}

The authors are grateful to the referee and the editor for several useful comments and suggestions.



\begin{thebibliography}{10}

\bibitem{Anderson-Hunter-Koch-1965}
L.~W.~Anderson, R.~P.~Hunter, and R.~J.~Koch,
\emph{Some results on stability in semigroups}.
Trans. Amer. Math. Soc. {\bf 117} (1965), 521--529.

\bibitem{Banakh-Bardyla-Guran-Gutik-Ravsky-2020}
T. Banakh, S. Bardyla, I. Guran, O. Gutik, and A. Ravsky,
\emph{Positive answers for Koch's problem in special cases},
Topol. Algebra Appl. \textbf{8} (2020),  76-87.

\bibitem{Bardyla-2016}
S. Bardyla,
\emph{Classifying locally compact semitopological polycyclic monoids},
Mat. Visn. Nauk. Tov. Im. Shevchenka \textbf{13} (2016), 21--28.

\bibitem{Bardyla-2018}
S. Bardyla,
\emph{On locally compact semitopological graph inverse semigroups},
Mat. Stud. \textbf{49} (2018), no. 1, 19--28.

\bibitem{Bardyla=2021??}
S. Bardyla,
\emph{On topological McAlister semigroups},
J. Pure Appl. Algebra \textbf{227} (2023), no.~4, 107274.

\bibitem{Bertman-West-1976}
M. O. Bertman and T. T. West,
\emph{Conditionally compact bicyclic semitopological semigroups},
Proc. Roy. Irish Acad. \textbf{A76} (1976), no. 21--23, 219--226.

\bibitem{Carruth-Hildebrant-Koch-1983}
J.~H.~Carruth, J.~A.~Hildebrant and  R.~J.~Koch,
\emph{The Theory of Topological Semigroups}, Vol. I, Marcel
Dekker, Inc., New York and Basel, 1983.

\bibitem{Carruth-Hildebrant-Koch-1986}
J.~H.~Carruth, J.~A.~Hildebrant, and  R.~J.~Koch,
\emph{The Theory of Topological Semigroups}, Vol. II, Marcel Dekker,
Inc., New York and Basel, 1986.

\bibitem{Chuchman-Gutik-2010}
I. Ya. Chuchman and O. V. Gutik,
\emph{Topological monoids of almost monotone injective co-finite partial selfmaps of the set of positive integers}.
Carpathian Math. Publ. \textbf{2} (2010), no.~1, 119--132.

\bibitem{Clifford-Preston-1961}
A. H.~Clifford and  G. B.~Preston,
\emph{The Algebraic Theory of Semigroups},
Vol. I., Amer. Math. Soc. Surveys 7, Pro\-vi\-den\-ce, R.I., 1961.

\bibitem{Clifford-Preston-1967}
A. H.~Clifford and  G. B.~Preston,
\emph{The Algebraic Theory of Semigroups},
Vol. II., Amer. Math. Soc. Surveys 7, Provi\-den\-ce, R.I., 1967.

\bibitem{Eberhart-Selden=1969}
C. Eberhart and J. Selden,
\emph{On the closure of the bicyclic semigroup},
Trans. Amer. Math. Soc. \textbf{144} (1969), 115--126.

\bibitem{Engelking-1989}
R.~Engelking,
\emph{General Topology},
2nd ed., Heldermann, Berlin, 1989.

\bibitem{Guran-Kisil-2012}
I. Guran and M. Kisil',
\emph{Pontryagin's alternative for locally compact cancellative monoids},
Visnyk Lviv Univ. Ser. Mech. Math. \textbf{77} (2012), 84--88 (in Ukrainian).

\bibitem{Gutik=2015}
O. Gutik,
\emph{On the dichotomy of a locally compact semitopological bicyclic monoid with adjoined zero},
Visnyk L'viv Univ., Ser. Mech.-Math. \textbf{80} (2015), 33--41.

\bibitem{Gutik-2018}
O. Gutik,
\emph{On locally compact semitopological $0$-bisimple inverse $\omega$-semigroups},
Topol. Algebra Appl. \textbf{6} (2018), 77--101.

\bibitem{Gutik-KhylynskyiP=2022}
O. Gutik and P. Khylynskyi,
\emph{On a locally compact monoid of cofinite partial isometries of $\mathbb{N}$ with adjoined zero},
Topol. Algebra Appl. \textbf{10} (2022), no. 1, 233--245.

\bibitem{Gutik-Maksymyk-2019}
O. V. Gutik and K. M. Maksymyk,
\emph{On a semitopological extended bicyclic semigroup with adjoined zero},
 J. Math. Sci. \textbf{265} (2022), no.~3, 369--381.

\bibitem{Gutik-Mykhalenych=2020}
O. Gutik and M. Mykhalenych,
\emph{On some generalization of the bicyclic monoid},
Visnyk Lviv. Univ. Ser. Mech.-Mat. \textbf{90} (2020), 5--19 (in Ukrainian).

\bibitem{Gutik-Mykhalenych=2021}
O. Gutik and M. Mykhalenych,
\emph{On group congruences on the semigroup $\boldsymbol{B}_{\omega}^{\mathscr{F}}$ and its homo\-mor\-phic retracts in the case when a family $\mathscr{F}$ consists of inductive non-empty subsets of~$\omega$},
Visnyk Lviv. Univ. Ser. Mech.-Mat. \textbf{91} (2021), 5--27 (in Ukrainian).

\bibitem{Hewitt-1956}
E. Hewitt,
\emph{Compact monothetic semigroups},
Duke Math. J. \textbf{23} (1956), no. 3, 447--457.


\bibitem{Hofmann-1960}
K. H.~Hofmann,
\emph{Topologische Halbgruppen mit dichter submonoger Untenhalbgruppe},
Math. Zeit. \textbf{74} (1960), 232--276.

\bibitem{Hofmann-Mostert-1966}
K. H. Hofmann and P. S. Mostert,
\emph{Elements of compact semigroups},
Co\-lum\-bus: Chas. E. Merrill Co., 1966.

\bibitem{Koch-1957}
R.~J.~Koch,
\emph{On monothetic semigroups},
Proc. Amer. Math. Soc. \textbf{8} (1957), no. 2, 397--401.



\bibitem{Lawson-1998}
M.~Lawson,
\emph{Inverse Semigroups. The Theory of Partial Symmetries},
Singapore: World Scientific, 1998.

\bibitem{Maksymyk-2019}
K. Maksymyk,
\emph{On locally compact groups with zero},
Visn. Lviv Univ., Ser. Mekh.-Mat. \textbf{88} (2019), 51--58.
(in Ukrainian).

\bibitem{Mokrytskyi-2019}
T. Mokrytskyi,
\emph{On the dichotomy of a locally compact semitopological monoid of order isomorphisms between principal filters of $\mathbb{N}^n$ with adjoined zero},
Visn. Lviv Univ., Ser. Mekh.-Mat. \textbf{87} (2019), 37--45.

\bibitem{Numakura-1952}
K. Numakura,
\emph{On bicompact semigroups},
Math. J. Okayama Univ. \textbf{1} (1952), 99--108.

\bibitem{Petrich-1984}
M.~Petrich,
\emph{Inverse Semigroups},
John Wiley $\&$ Sons, New York, 1984.

\bibitem{Ruppert-1984}
W.~Ruppert,
\emph{Compact Semitopological Semigroups: An Intrinsic Theory},
Lect. Notes Math., \textbf{1079}, Springer, Berlin, 1984.

\bibitem{Wagner-1952}
V.~V. Wagner,
\textit{Generalized groups},
Dokl. Akad. Nauk SSSR \textbf{84} (1952), 1119--1122 (in Russian).

\bibitem{Weil-1938}
A. Weil.
\textit{L'integration dans les groupes lopologiques et ses applications},
Actualites Scientifiques No. 869, Hermann, Paris, 1938.

\bibitem{Zelenyuk-1988}
E. G. Zelenyuk,
\textit{On Pontryagin's alternative for topological semigroups},
Mat. Zametki \textbf{44} (1988), no. 3, 402--403 (in Russian).

\bibitem{Zelenyuk-2019}
Ye. Zelenyuk,
\emph{A locally compact noncompact monothetic semigroup with identity},
Fund. Math. \textbf{245} (2019), no. 1, 101--107.

\bibitem{Zelenyuk-2020}
Ye. Zelenyuk,
\emph{Larger locally compact monothetic semigroups},
Semigroup Forum \textbf{100} (2020), no. 2, 605--616.

\bibitem{Zelenyuk-Zelenyuk-2020}
Ye. Zelenyuk and Yu. Zelenyuk,
\emph{When a locally compact monothetic semigroup is compact},
J. Group Theory \textbf{23} (2020), no. 6, 983--989.


\end{thebibliography}
\end{document}